\documentclass{amsart}

\usepackage{amsmath}
\usepackage{amsfonts}
\usepackage{amssymb}
\usepackage{amsthm}
\usepackage{mathrsfs}

\newtheorem{theorem}{Theorem}[section]
\newtheorem{theoreml}{Theorem}[section]

\newtheorem{lemma}[theorem]{Lemma}
\newtheorem{proposition}[theorem]{Proposition}
\newtheorem{corollary}[theorem]{Corollary}

\newtheorem{question}{Question}

\newtheorem{definition}[theorem]{Definition}

\theoremstyle{remark}

\newtheorem*{claim}{Claim}

\newcommand{\cof}[1]{\mathrm{cof}(#1)}
\newcommand{\add}[1]{\mathrm{add}(#1)}
\newcommand{\non}[1]{\mathrm{non}(#1)}

\newcommand{\cov}[1]{\mathrm{cov}(#1)}
\newcommand{\covstar}[1]{\mathrm{cov}^*(#1)}
\newcommand{\nonstar}[1]{\mathrm{non}^*(#1)}
\newcommand{\forces}[1]{\Vdash``#1"}
\newcommand{\baire}{\omega^\omega}
\newcommand{\bairetree}{\omega^{<\omega}}
\newcommand{\Baire}{[\omega]^\omega}

\newcommand{\ros}{\mathfrak{ros}}

\newcommand{\dom}{\mathfrak{d}}
\newcommand{\unb}{\mathfrak{b}}
\newcommand{\spl}{\mathfrak{s}}
\newcommand{\SSS}{{\mathcal S}}
\newcommand{\FF}{{\mathcal F}}
\newcommand{\sub}{\subseteq}
\newcommand{\GG}{{\mathcal G}}
\newcommand{\II}{{\mathcal I}}
\newcommand{\JJ}{{\mathcal J}}
\newcommand{\N}{\mathbb{N}}
\newcommand{\PP}{\mathcal{P}}
\newcommand{\MM}{\mathcal{M}}
\newcommand{\NN}{\mathcal{N}}
\newcommand{\EE}{\mathcal{E}}
\newcommand{\DD}{\mathcal{D}}

\newcommand{\AAA}{\mathcal{A}}

\title[Cardinal invariants of Rosenthal families and large-scale topology]{On cardinal invariants related to Rosenthal families and large-scale topology}
\author{Arturo Mart\'{i}nez-Celis and Tomasz Żuchowski}
\address{Instytut Matematyczny, Uniwersytet Wrocławski, pl. Grunwaldzki 2, 50-384 Wrocław, Poland}
\email{arturo.martinez-celis@math.uni.wroc.pl}
\address{Instytut Matematyczny, Uniwersytet Wrocławski, pl. Grunwaldzki 2, 50-384 Wrocław, Poland}
\email{tomasz.zuchowski@math.uni.wroc.pl}


\begin{document}
\begin{abstract}
Given a function $f \in \baire$, a set $A \in [\omega]^\omega$ is \emph{free for} $f$ if $f[A] \cap A$ is finite. For a class of functions $\Gamma\sub\omega^{\omega}$, we define $\mathfrak{ros}_\Gamma$ as the smallest size of a family $\mathcal{A}\sub[\omega]^\omega$ such that for every $f\in\Gamma$ there is a set $A \in \AAA$ which is free for $f$, and $\Delta_\Gamma$ as the smallest size of a family $\mathcal{F}\sub\Gamma$ such that for every $A\in[\omega]^\omega$ there is $f\in\mathcal{F}$ such that $A$ is not free for $f$. We compare several versions of these cardinal invariants with some of the classical cardinal characteristics of the continuum. Using these notions, we partially answer some questions from \cite{ArturoKoszmider} and \cite{BanakhProtasov}.
\end{abstract}

\maketitle
\section{Introduction}
A matrix $M=(m_{k, n})_{k, n\in \omega}$ consisting of non-negative reals is called a \textit{Rosenthal matrix} if the set of sums $\{\sum_{n\in \N}m_{k, n}: k\in \omega\}$ is bounded. Given $\varepsilon>0$ and a Rosenthal matrix $M$, an infinite subset $A$ of $\omega$ \textit{$\varepsilon$-fragments} $M$ if for every $k\in A$ we have
$$\sum_{n\in A\setminus \{k\}}m_{k, n} < \varepsilon.$$
A family $\mathcal F$ of infinite subsets of $\omega$ is a \textit{Rosenthal family} (\cite[Definition 1.3]{damian-rosenthal}) if for any $\varepsilon>0$ and every Rosenthal matrix $M$ there is $A\in\FF$ such that it $\varepsilon$-fragments $M$.

The Rosenthal's lemma (\cite{rosenthal1}, \cite{rosenthal2}), an important result concerning sequences of measures (see \cite[Section 6]{ArturoKoszmider} for the discussion of
its classical forms and uses), states in its combinatorial form that $\Baire$ is a Rosenthal family (see e.g. \cite{kupka} and \cite[Section 1]{damian-rosenthal}).

Rosenthal families have been recently studied in \cite{damian-rosenthal}, \cite{ArturoKoszmider} and \cite{Repicky}. In \cite{ArturoKoszmider}, the authors proved that every ultrafilter on $\omega$ is a Rosenthal family, and that $\mathfrak{r}$ is the smallest size of a Rosenthal family. Moreover, they found out a connection of Rosenthal families with free sets.

We say that a set $A\sub\Baire$ is \textit{free} for a given $f\in\baire$ if $f[A]\cap A=\emptyset$. It follows a well-established
combinatorial terminology (e.g., \cite{Komjath}) in which, given a set mapping $f\colon X \to\PP(X)$, a set $Y \sub X$ is called free if $y\notin f(y')$ for any $y,y'\in Y$.  It was proven in \cite[Theorem 23]{ArturoKoszmider} that $\mathfrak{r}$ is the smallest size of a family $\AAA\sub\Baire$ such that for every $f\colon\omega\to\omega$ without fixed points there is an $A\in\AAA$ which is free for $f$. Functions without fixed points correspond to Rosenthal matrices of zeros and ones (see \cite[Section 2.2]{ArturoKoszmider}). 

Many of the classical cardinal invariants can be investigated using relational systems, introduced by Vojt\'{a}\v{s} in \cite{vojtas} (see also \cite[Section 4]{Blass}). By a \textit{relational system} we mean a triple $\big\langle A,B,R\big\rangle$, where $R\sub A\times B$ is a binary relation on non-empty sets $A$ and $B$. For $x\in A$ and $y\in B$, $x\, R\, y$ is read as $y$ $R$\emph{-dominates} $x$. 
A family $X\sub A$ is \textit{$R$-unbounded} if there is no element of $B$ that $R$-dominates every element of $X$, and $Y\sub B$ is $R$\emph{-dominating} if every element of $A$ is $R$-dominated by some element of $Y$. The cardinal $\dom(R)$ is the smallest size of an $R$-dominating family, and $\unb(R)$ is the smallest size of an $R$-unbounded family. For example, for the relation $\mathfrak{D}=\langle\baire, \baire, \leq^*\rangle$ we have $\dom(\mathfrak{D})=\dom$ and $\unb(\mathfrak{D})=\unb$. Another common example is the relation $\mathfrak{R}=\langle \PP(\omega), \Baire, \textrm{does not split}\rangle$, for which we have $\dom(\mathfrak{R})=\mathfrak{r}$ and $\unb(\mathfrak{R})=\spl$.

The previously stated characterization of $\mathfrak{r}$ by free sets yields the following natural relational system:
If we denote by $f\, \mathcal{R}\, A$ whenever $A$ is \textit{free} for $f$, then, the dominating number of this relational system is $\mathfrak{r}$.  Observe that $\unb(\mathcal{R})=\omega$. Indeed, for every $a\neq b\in\omega$ there is $f_{a,b}\in\baire$ without fixed points such that $f(a)=b$ and $f(b)=a$. The family $\{f_{a,b}\colon\{a,b\}\in[\omega]^2\}$ is $\mathcal{R}$-unbounded. Instead, we will look at the following modification of the relation $\mathcal{R}$.

\begin{definition}
Let $f\in\baire$ and $A\sub\Baire$. We say that $A$ is \textit{*-free} for $f$ if $f[A]\cap A$ is finite.   
\end{definition}

As with the relation $\mathcal{R}$, the notion of *-free yields a natural relation $R'$: $f \mathcal{R}' A$ if and only if $A$ is *-free for $f$. One can easily show that $\dom(\mathcal{R}')=\dom(\mathcal{R})$ and that $\unb(\mathcal{R}')$ is uncountable. In \cite{ArturoKoszmider}, the authors also consider a restriction of this relation, limiting the domain to finite-to-one functions in $\baire$. They proved that the dominating number of this restriction, which they call $\ros(c_0)$, is equal to $\min\{\mathfrak{d},\mathfrak{r}\}$ \cite[Theorem 32]{ArturoKoszmider}, and asked about the case when the relation is restricted to only injective functions (\cite[Question 3]{ArturoKoszmider}).

To investigate the mentioned question, we will consider the following scheme of notions. 
\begin{definition}    
Given a property $\varphi$ about functions of $\baire$, let
$$
\square_\varphi = \big\{ f \in \baire : f \text{ has no fixed points and }\varphi(f) \big\},
$$
and let $\ros_\varphi=\mathfrak{d}\big(\big\langle\square_\varphi, \Baire, \textrm{*-free}\big\rangle\big)$. In other words
$$
\ros_\varphi = \min \big\{|\AAA| :\AAA \subseteq \Baire \text{ and }\: \forall f \in \square_\varphi \:\exists A \in \mathcal{A} \:\: \big( f[A]\cap A \text{ is finite} \big) \big\}.
$$
\end{definition}

For example, if $\varphi$ is the trivial property, i.e., a property that all function satisfy (which, from now on, will be denoted by $\varphi = \textbf{1}$), or if $\varphi$ is the property of being a finite to one function (from now on $\varphi = \text{Fin-1}$), then the results from \cite{ArturoKoszmider} can be restated as follows:
$$\mathfrak{ros}_{\textbf{1}}=\mathfrak{r} \textrm{ and } \mathfrak{ros}_{\text{Fin-1}}=\min\{\mathfrak{d},\mathfrak{r}\}.$$

If $\varphi$ is the property of being a one-to-one function ($\varphi = 1-1$) then \cite[Question 3]{ArturoKoszmider} asks about the value of $\mathfrak{ros}_{\text{1-1}}$. In this work we prove the following result.
\begin{theoreml}
$\cov{\MM}\leq \mathfrak{ros}_{\text{1-1}}\leq \min \big\{ \mathfrak{d}, \mathfrak{r}, \min_{g\in\baire} \widehat{\mathfrak{e}_g}\big\}$
and it is consistent with ZFC that $\mathfrak{ros}_{\text{1-1}} < \min \big\{ \mathfrak{d}, \mathfrak{r}, \non{\NN}\big\}$.
\end{theoreml}
The cardinals $\widehat{\mathfrak{e}_g}$ are related to the notion of being eventually different inside a compact set (the precise definition will be on Section $4$). The proof follows from Theorems \ref{banakhbounds}, \ref{eg_bounds} and \ref{ros_mathias}.

We also show a close relationship between the relation of being *-free and some of the cardinal invariants considered by T. Banakh in \cite{Banakh}, in which the author studies some set-theoretical problems related to \textit{large-scale topology}. Large-scale topology (referred also as Asymptology \cite{asymptology}) studies properties of \textit{coarse spaces}, which were introduced independently in \cite{coarse1} and \cite{coarse2}. The cardinal $\Delta$ is defined as the smallest weight of a finitary coarse structure on $\omega$ that contains no infinite discrete subspaces (for the definitions and more information about the subject, we refer the reader to \cite{BanakhProtasov}). In \cite{Banakh}, the author found the following two combinatorial characterizations of $\Delta$, where $S_{\omega}$ denotes the permutation group of $\omega$ and $I_\omega\sub S_\omega$ is the set of involutions of $\omega$ (i.e. permutations that are their own inverse) that have at most one fixed point. 

\begin{theorem}[Banakh]\label{involutions}
\begin{align*}
        \Delta &= \min\big\{|\FF|:\FF\subseteq S_{\omega}\;\text{ and }\; \forall A\in\Baire\;\exists f\in\FF\;\;\big(|\{x\in A:x\ne f(x)\in A\}| = \aleph_0 \big)\big\}\\
&= \min\big\{|\FF|:\FF\subseteq I_{\omega}\;\text{ and }\; \forall A\in\Baire\;\exists f\in\FF\;\;\big( |f[A] \cap A| = \aleph_0\big)\big\}.
\end{align*}    

\end{theorem} 
The author also considered the dual cardinal invariant, denoted by $\widehat{\Delta}$:
$$ \widehat{\Delta}=\min\big\{|\AAA|:\AAA\subseteq\Baire\;\text{ and }\;\forall h\in I_\omega\;\;\exists A\in\AAA\;\;(|h[A]\cap A|<\aleph_0)\big\}. $$

In \cite[Section 3 and 7]{Banakh}, the author relates $\Delta$ and $\widehat{\Delta}$ with other classical cardinal invariants of the continuum:

\begin{theorem}[Banakh] \label{maintheoremBanakh}
    $$
\max \{ \mathfrak{b}, \mathfrak{s}, \cov{\NN} \} \leq \Delta \leq \non{\MM}, $$
$$
\cov{\MM} \leq \widehat{\Delta} \leq \min \{ \mathfrak{d}, \mathfrak{r} ,\non{\NN}\}.
    $$
\end{theorem}

For the following theorem, the cardinals ${\mathfrak{e}_g}$ are related to the notion of being infinitely equal inside a compact set (the precise definition will be on Section $4$).

\begin{theoreml}\label{theoremb}
We have $\sup_{g\in\baire} \mathfrak{e}_g\leq \Delta$ and $\widehat{\Delta} \leq \min_{g\in\baire} \widehat{\mathfrak{e}_g}$. It is consistent with ZFC that $\max \big\{ \mathfrak{b}, \mathfrak{s}, \cov{\NN}\big\} < \Delta,$ and it is consistent with ZFC that
$\widehat{\Delta} < \min \big\{ \mathfrak{d}, \mathfrak{r}, \non{\NN}\big\}$.
\end{theoreml}

 This partially answers \cite[Problem 2.3]{BanakhProtasov} and \cite[Problem 2.5]{BanakhProtasov}. The proof follows from Corollary \ref{deltabanakhdelta} and Theorems \ref{eg_bounds}, \ref{ros_mathias}, \ref{delta_mathias}.

In order to prove this theorem, we will study the dual cardinal invariants of $\ros_\varphi$, which will be called $\Delta_\varphi$. These are defined in the following way.
\begin{definition}
Given a property $\varphi$ about functions of $\baire$, let
$\Delta_\varphi=\mathfrak{b}\big(\big\langle\square_\varphi, \Baire, \textrm{*-free}\big\rangle\big)$. In other words,
$$\Delta_\varphi = 
\min \big\{|\FF| : \FF \subseteq \square_\varphi \text{ and } \: \forall A \in \Baire \: \exists f \in \FF \:\: \big( |f[A] \cap A| = \aleph_0\big) \big\}.
$$
\end{definition}

Clearly, we have $\ros_{\text{1-1}} \leq \ros_{\text{Fin-1}} \leq \ros_\textbf{1}$ and $\Delta_\textbf{1} \leq \Delta_{\text{Fin-1}} \leq \Delta_\text{1-1}$. These families of cardinal invariants are closely connected with $\Delta$ and $\widehat{\Delta}$: by Corollary \ref{deltabanakhdelta} we have $\Delta = \Delta_{bijective}=\Delta_\text{1-1}$ and $\widehat\Delta = \ros_\text{bijective} =\ros_\text{1-1}$. The part about consistency results in Theorem \ref{theoremb} about $\Delta$ follow from the following result, proved in Theorems \ref{eg_bounds} and \ref{delta_mathias}.

\begin{theoreml}
It is true that $\sup_{g\in\baire} \mathfrak{e}_g\leq \Delta_\text{1-1}$ and 
it is consistent with ZFC that $\max \big\{ \mathfrak{b}, \mathfrak{s}, \sup_{g\in\baire} \mathfrak{e}_g\big\} < \Delta_\text{1-1}.$
\end{theoreml}

The structure of the paper is as follows: In Section $2$ we prove the lower and upper bounds for the invariants $\Delta_\textbf{1}$ and $\Delta_{\text{Fin-1}}$, using similar ideas than the ones in \cite{ArturoKoszmider} for $\ros_\textbf{1}$ and $\ros_{\text{Fin-1}}$. Section $3$ is devoted to prove that $\Delta$ and $\widehat{\Delta}$ is equal to a certain class of $\Delta_\varphi$ and $\ros_\varphi$, respectively, with $\varphi$ as the property of being injective being the most representative example of that class. In Section $4$, we relate $\Delta_{\text{1-1}}$ and $\ros_\text{1-1}$ to some cardinal invariants closely related to infinitely equal and eventually different reals in some compact sets. In the last section, we provide some consistency results concerning $\Delta_{\text{1-1}}$ and $\ros_\text{1-1}$.

Our notation is standard and follows \cite{Bartoszynski} and \cite{Blass}. The reader can find the definitions of the classic cardinal invariants $\dom$, $\unb$, $\spl$, $\mathfrak{r}$, $\cov{\MM}$, $\non{\MM}$, $\add{\MM}$, $\cof{\MM}$, $\cov{\NN}$ and $\non{\NN}$ in \cite{Blass}.

\section{$\Delta_\textbf{1}$ and $\Delta_{\text{Fin-1}}$.}

In \cite{ArturoKoszmider}, the authors prove that $\ros_\textbf{1} = \mathfrak{r}$ and that $\ros_\text{Fin-1} = \min \{ \dom, \mathfrak{r}\}$. In this section, we will prove a similar version, but not exactly the same, of the dual of those results. One of the main tools that it is used to prove that $\ros_\textbf{1} \geq \mathfrak{r}$ is the fact that, if $\kappa < \mathfrak{r}$ and $\mathcal{S} \subseteq \Baire$ is a family of sets such that $|\mathcal{S}| = \kappa$, then there is a partition $\omega = \bigcup_{n \in \omega} P_n$ such that for each $S \in \mathcal{S}$ and each $P_n$, $S \cap P_n$ is infinite, and $\mathfrak{r}$ is the smallest cardinal to have such property (which the authors of \cite{ArturoKoszmider} call nowhere reaping). It is not clear that the dual of that property is $\spl$. This issue was already addressed in \cite{MixReals}. In the Section 8 of that work, the authors introduce a cardinal invariant, closely related to $\spl$: Let $\mathbf{P}$ be the family of partitions of $\omega$ into $\aleph_0$ many infinite sets.
$$
\spl_\text{mix} = \min \{ |\mathcal{S}| : \mathcal{S} \subseteq \mathbf{P} \text{ and }\forall A\in[\omega]^{\omega}\: \exists P=(P_k)_{k\in\omega}\in\SSS\: \forall k\in\omega \: (|A\cap P_k|=\aleph_0) \}
$$

Clearly $\mathfrak{s} \leq \spl_\text{mix}$. By the remark on the last paragraph, the dual of $\spl_\text{mix}$ is $\mathfrak{r}$, however, it is not known if $\spl_\text{mix} = \spl$. The proof of the following proposition implicitly appears in \cite[Proposition 7.3.(iii)]{MixReals}.

\begin{proposition}[Farkas, Khomskii, Vidny\'{a}nszky] \label{splmixnonm}
    $\spl_\text{mix}\leq\non{\MM}$
\end{proposition}
\begin{proof}
    Consider the space $P = \{ f \in \baire : \forall n \in \omega \: (f^{-1}(n) \text{ is infinite}) \}$ of partitions of $\omega$ into $\omega$ many infinite sets. Notice that $P \subseteq \baire$ is Polish. Let $\kappa < \spl_\text{mix}$ and let $\mathcal{F} = \{ f_\alpha : \alpha \in \kappa\} \subseteq P$. Then, there must be an infinite set $A \subseteq \omega$ such that for all $\alpha \in \kappa$ there is an $n\in\omega$ such that $f^{-1} (n) \cap A$ is finite. For every $k,n\in\omega$ the set $E_n^k=\{ f \in P : |f^{-1}(n) \cap A| \leq k  \}$ is a closed nowhere dense subset of $P$. It follows that  $\mathcal{F}$ is meager, since we have $\mathcal{F} \subseteq \bigcup_{n,k \in \omega} E_n^k$.
\end{proof}

We prove that this invariant is an upper bound for $\Delta_\textbf{1}$. The proof is analogous to the proof of the dual inequality (see \cite[Lemma 17]{ArturoKoszmider}).

\begin{proposition}
$\Delta_\textbf{1}\leq\spl_\text{mix}$.
\end{proposition}
\begin{proof}
Let $\SSS_{\omega}$ be a family of partitions of $\omega$ into $\omega$ many infinite sets witnessing the definition of $\spl_\text{mix}$. We will build a set of functions $\mathcal{F}\subseteq \square$ such that for every $A \in \Baire$ there is $f \in \mathcal{F}$ such that $f[A]\cap A$ is infinite. For every partition $P=(P_n)_{n\in\omega}\in\SSS_{\omega}$ we define the following function $f_P\colon\omega\to\omega$:
\[ f_P(k) = \begin{cases}
	n, & \text{if } n\neq k\in P_n, \\
	k+1, & \text{otherwise.}
\end{cases} \]
It follows that $f_P$ is a function without fixed points. Let $\mathcal{F} = \{ f_P : P \in \SSS_\omega \}$. If $A\in[\omega]^{\omega}$ then there exists a $P=(P_n)_{n\in\omega}\in\SSS_{\omega}$ such that $\big|A\cap P_n\big|=\omega$ for all $n\in\omega$.  Then, for every $n\in\omega$ we have:
\[f_P^{-1}\big[\{n\}\big]\supseteq \big(P_n\backslash\{n\}\big) \cap A \neq\emptyset, \]
which implies $n\in f_P[A]$ for all $n\in\omega$, and so $f_P[A]=\omega$. In particular, $f_P[A] \cap A$ is infinite.
\end{proof}

For our next result, we will use the following well-known theorem, whose proof can be found in \cite{Katetov}. 

\begin{theorem}
    For every $f \in \square$ there is a partition $\omega = A_0\cup A_1 \cup A_2$ such that $f[A_i]\cap A_i = \emptyset$ for all $i \in \{0,1,2\}$.
\end{theorem}

The following is a direct corollary from the previous theorem and the fact that $\spl$ is the smallest size of a family of partitions into $3$ sets such that for every infinite $X \subseteq \omega$ there is a partition in the family such that every part intersects $X$ in an infinite set (see e.g. \cite[Section 8]{MixReals}).

\begin{proposition} \label{slessdelta}
$\spl\leq\Delta_\textbf{1}$.
\end{proposition}
\begin{proof}
Let $\FF=\{f_\alpha : \alpha<\kappa\}\subseteq \square$ with $\kappa < \mathfrak{s}$. We will see that there is an infinite $A \in \Baire$ such that $f_\alpha[A] \cap A$ is finite for every $\alpha \in \kappa$. For each $\alpha<\kappa$ pick a partition $\omega = A^\alpha_0\cup A^\alpha_1 \cup A^\alpha_2$ such that $f_\alpha\big[A^\alpha_i\big]\cap A^\alpha_i=\emptyset$ for $i \in \{0,1,2 \}$. 

The family $\{ A^\alpha_i : \alpha \in \kappa \}$ is not a splitting family, so there must be an infinite $B\sub\omega$ such that, for every $\alpha<\kappa$, we have that $B\sub^* A^\alpha_{i_B}$ for exactly one $i_B \in \{0,1,2\}$. Therefore, for all $\alpha \in \kappa$, $f_\alpha[B] \cap B \subseteq^* f_\alpha[A^\alpha_{i_B}] \cap A^\alpha_{i_B} = \emptyset$.
\end{proof}

As a consequence, we get that $\Delta_\textbf{1}$ is in between $\spl$ and $\spl_\text{mix}$, raising the following question:

\begin{question}
    Is it true that $\spl = \Delta_\textbf{1} = \spl_\text{mix}$?
\end{question}

 We will now turn our focus on $\Delta_\text{Fin-1}$: For our next proposition, we will need a well-known characterization of $\unb$. Let $\II $ and $\JJ$ be interval partitions of $\omega$. The partition $\II$ \emph{dominates} the partition $\JJ$ if for almost all $I \in \II$ there is a $J \in \JJ$ such that $J \subseteq I$. The following widely known theorem appears in \cite[Theorem 2.10]{Blass}:

\begin{theorem}
$\unb$ is the smallest amount of interval partitions of $\omega$ that are not dominated by a single interval partition. $\dom$ is the smallest size of a dominating family of interval partitions. \label{blassdominating}
\end{theorem}

By taking the end points of every piece, we can consider interval partitions as infinite subsets of $\omega$, and the domination relations translates to the following: A set $A \in \Baire$ dominates a set $B \in \Baire$ if and only if for almost every $n \neq m \in A$ there are $i \neq j \in B$ such that $n \leq i < j \leq m$. With this in mind, we will prove the following proposition.

\begin{proposition}
$\unb\leq\Delta_{\text{Fin-1}}$. \label{blessdelta}
\end{proposition}
\begin{proof}
Let $\FF=\{f_\alpha\colon\alpha<\kappa\}\sub\square_\text{Fin-1}$ with $\kappa<\unb$. For every $\alpha<\kappa$, one can easily construct an increasing function $h_\alpha \in \baire$ such that $$h_\alpha(n+1) > \max\Big(f_\alpha[\{0,\ldots,h_\alpha(n)\}] \cup f^{-1}_\alpha[\{0,\ldots,h_\alpha(n)\}]\Big).$$ For any $\alpha<\kappa$ we define an interval partition $\II_\alpha=\big\{[h_\alpha(i), h_\alpha(i+1))\colon i\in\omega\big\}$. As $\kappa<\unb$, there must be an infinite set $A \in \Baire$ such that if $A=\{a_i : i \in \omega\}$ is its increasing enumeration, then, for every $\alpha< \kappa$ and for almost all $n\in \omega$, $[h_\alpha(k), h_\alpha(k+1)) \subseteq [a_n, a_{n+1})$ for some $k\in \omega$. We will see that $f_\alpha[A]\cap A$ is finite for every $\alpha \in \kappa$: Notice that, since for almost all $n \in \omega$ there is $k \in \omega$ such that $[h_\alpha(k), h_\alpha(k+1)) \subseteq [a_n, a_{n+1})$, then $a_{n+1} \geq h_\alpha(k + 1)$, implying that $a_{n+1}$ is bigger than any number of the form $f_\alpha(\ell)$ or an element of the set $f^{-1}_\alpha(\ell)$ for $\ell\leq h_\alpha(k)$. In particular, for almost all $n\in\omega$ we have $a_{n+1} > f_\alpha(a_i), \max f^{-1}_\alpha(a_i)$ for $i \leq n$. This implies that $f_\alpha[A] \cap A$ is finite. 
\end{proof}

Therefore, as a conclusion we get that $\max \{ \unb, \Delta_\textbf{1} \} \leq \Delta_{\text{Fin-1}}$. We will now prove the converse inequality, whose proof is inspired by the proof of the dual inequality in \cite[Proposition 26]{ArturoKoszmider}.

\begin{proposition}
$\Delta_{\text{Fin-1}}\leq\max\{\unb, \Delta_\textbf{1}\}$.
\end{proposition}
\begin{proof}
Let $\GG=\{g_\beta\colon\beta<\Delta_\textbf{1}\}\sub\square$ be a family of functions such that for every $A\in[\omega]^{\omega}$ there is $g\in\GG$ satisfying $\big|g[A]\cap A\big|=\omega$ and let $\{ A_\alpha : \alpha \in \unb\} \subseteq \Baire$ such that for every $B \in \Baire$ there is $\alpha \in \unb$ such that for almost every $n \neq m \in A$ there are $i \neq j \in B$ such that $n \leq i < j \leq m$. For every $\alpha \in \kappa$ consider $A_\alpha = \{ a^\alpha_n : n \in \omega \}$ with its increasing enumeration.

Let us define $f_{\alpha,\beta}\colon\omega\to\omega$ for each $\alpha<\unb$ and $\beta<\Delta_\textbf{1}$ by setting:
\[ f_{\alpha,\beta}(i) = \begin{cases}
g_{\beta}(i) & \text{ if there is an }n \in \omega \text{ such that } i,g_{\beta}(i)\in[a^\alpha_n,a^\alpha_{n+1}),\\
i+1 & \text{otherwise}.
\end{cases} \]
Clearly $f_{\alpha,\beta} \in \square_{\text{Fin-1}}$. We will show that the family $\FF=\{f_{\alpha,\beta} : \alpha \in \unb, \beta \in \Delta_\textbf{1} \}$ is a witness for $\Delta_\text{Fin-1}$. Let $A\in[\omega]^{\omega}$. Then, there is $\beta \in \Delta_\textbf{1}$ such that $\big|g_{\beta}[A]\cap A\big|=\omega$.
It is easy to find a set $B = \{ b_j : j \in \omega \} \in \Baire$ such that for every $i \in \omega$ we have $$g_{\beta}[[b_i,b_{i+1}) \cap A]\cap (A\cap [b_i,b_{i+1})) \neq \emptyset.$$ Then, find $\alpha < \unb$ such that there are infinitely many $n \neq m \in A$ such that there are $i \neq j \in B$ such that $n \leq i < j \leq m$. Clearly, this implies that $$g_{\beta}[[a^\alpha_n,a^\alpha_{n+1}) \cap A]\cap (A\cap [a^\alpha_n,a^\alpha_{n+1})) \neq \emptyset$$ holds for infinitely many $n\in\omega$. Finally, note that if $$i \in g_{\beta}[[a^\alpha_n,a^\alpha_{n+1}) \cap A]\cap (A\cap [a^\alpha_n,a^\alpha_{n+1})),$$ then $i, g_\beta(i) \in [a^\alpha_n,a^\alpha_{n+1})$ and therefore $f_{\alpha,\beta}(i) = g_\beta (i)$. In conclusion, the set $f_{\alpha,\beta}[A] \cap A$ is infinite.
\end{proof}

As a corollary, we get the main result of this section.

\begin{theorem} \label{deltafin1charact}
    $\Delta_{\text{Fin-1}}=\max\{\unb, \Delta_\textbf{1}\}$.
\end{theorem}

This implies that $\max\{\unb, \spl \} \leq \Delta_{\text{Fin-1}} \leq \max\{\unb, \spl_{\text{mix}}\}$. The dual version of this theorem was proven in \cite[Section 5]{ArturoKoszmider}, where the authors proved that $\ros_{\text{Fin-1}} = \min \{ \dom, \ros_\textbf{1}\}$, and since $\mathfrak{r} = \ros_\textbf{1}$, then $\ros_{\text{Fin-1}} = \min \{ \dom, \mathfrak{r}\}$. We were not able to prove the dual version this last equality, leaving us with the following question.

\begin{question}
    Is it true that $\max\{\unb, \spl \} = \Delta_{\text{Fin-1}} = \max\{\unb, \spl_{\text{mix}}\}$?
\end{question}

\section{$\Delta_\text{1-1}$, $\ros_\text{1-1}$ and its relatives.}

In this section, we will take a look at $\Delta_\text{1-1}$, $\Delta_\text{bijective}$, $\Delta_\text{involution}$ and its dual versions (where $\varphi =$ bijective is the property of being a bijective function and $\varphi =$ involution is the property of being an involution). Since $\square_\text{involution} \subseteq \square_\text{bijective} \subseteq \square_\text{1-1} \subseteq \square_\text{Fin-1} \subseteq \square_\textbf{1}$, then $\Delta_\text{involution} \geq \Delta_\text{bijective} \geq \Delta_\text{1-1} \geq \Delta_\text{Fin-1} \geq \Delta_\textbf{1}$ and $\ros_\text{involution} \leq \ros_\text{bijective} \leq \ros_\text{1-1} \leq \ros_\text{Fin-1} \leq \ros_\textbf{1}$ and as a consequence, $\max\{\unb, \spl\} \leq \Delta_\text{1-1}$ and $\min\{\dom, \mathfrak{r}\} \geq \ros_\text{1-1}$. The main goal of this section is to prove that $\Delta_\text{1-1} = \Delta_\text{involution} = \Delta$ and that $\ros_\text{1-1} = \ros_\text{involution} = \widehat{\Delta}$.

Before that, we must point out an easy remark about involutions that will be used frequently: Involutions without fixed points may not exist, for example when the domain has an odd amount of points. However, if the domain has an even amount of points, or an infinite amount of points, it is always possible to construct an involution. With this in mind, we will proceed to prove the following lemma.

\begin{lemma}
    For every $f \in \square_{1-1}$ there are $\hat{f}_0,\hat{f}_1,\hat{f}_2,\hat{f}_3 \in \square_{\text{involution}}$, for every $A \in \Baire$, if all $f_i[A] \cap A$ are finite, then $f[A] \cap A$ is finite. 
\end{lemma}
\begin{proof} First observe that if $f$ and $f'$ are such that $f(k) \neq f'(k)$ for finitely many $k\in \omega$ then $f[A] \cap A$ is finite if and only if $f'[A] \cap A$ is finite. Define an equivalence relation $R$ on $\omega$: $m R n$ if and only if there is an $i\in \mathbb{Z}$ such that $m = f^i(n)$. This splits $\omega$ into (possibly finitely many) classes $\{C_n: n \in N\}$. We may categorize these $C_n$ in the following way:
\begin{itemize}
    \item $C_n$ is finite. In this case $C_n$ can be enumerated as $\{ a^n_i : i \leq k_n \}$ such that $f(a_{k_n}^n) = a_0^n$ and for every $i < k_n, f(a_i^n) = a_{i+1}^n$. Note that if $C_n$ is finite, then $C_n$ has an odd number of elements, if and only if $k_n$ is even. 
    \item $C_n$ is infinite. We will distinguish two subcases:
    \begin{itemize}
        \item $C_n$ is a cycle, i.e. for all $i\in C_n$ there is $j\in C_n$ such that $f(j) = i$. In this case $C_n$ can be enumerated as $\{ a^n_i : i \in \mathbb{Z} \}$ such that for every $i < \mathbb{Z}, f(a_i^n) = a_{i+1}^n$.
        \item $C_n$ is a \emph{semi-cycle}, i.e. there is an $i\in C_n$ such that for all $j\in C_n f(j) \neq i$. Notice that this $i$ must be unique. In this case $C_n$ can be enumerated as $\{ a^n_i : i \in \omega \}$ such that for every $i < \mathbb{Z}, f(a_i^n) = a_{i+1}^n$.
    \end{itemize}
\end{itemize}
Notice that, if $x,y$ are such that $f(x) = y$, then there is an $n \in \omega$ such that either $x=a^n_i$ and $y = a^n_{i+1}$ for some $i < k_n$ (where $k_n = \infty$ in case that $C_n$ is infinite), or $x=a^n_{k_n}$ and $y = a^n_{0}$. We have to split the proof into three cases:

\medskip
\underline{\textbf{Case 1.}} There is an even number of $C_n$ such that $|C_n|$ is odd, or there is an odd number of $C_n$ such that $|C_n|$ is odd and there is an infinite $C_k$. 

Note that, for such $f$, there is an $f' \in \square_{1-1}$ such that $f(k) \neq f'(k)$ for finitely many $k\in \omega$ and $f'$ has no $C_n$ with an odd number of elements. So, without loss of generality, we may assume that $f$ itself has no $C_n$ with an odd number of elements. We will construct $\hat{f}_0,\hat{f}_1,\hat{f}_2,\hat{f}_3 \in \square_{\text{involution}}$ by defining them in each restriction $\hat{f}_i \restriction C_n$:
\begin{itemize}
    \item $C_n$ has an even number of elements. In this case, let $\hat{f}_0 (a^n_i) = a^n_{i+1}$ and $\hat{f}_0 (a^n_{i+1}) = a^n_i$ whenever $i < k_n$ is even, let $\hat{f}_1 (a^n_i) = a^n_{i+1}$, $\hat{f}_0 (a^n_{i+1}) = a^n_i$ when $i \leq k_n$ is odd (we convene that $a^n_{k_n+1} = a^n_0$). Define $\hat{f}_2 \restriction C_n,\hat{f}_3 \restriction C_n$ to be any random involutions on $C_n$. It is clear that each $\hat{f}_i$ is an involution without fixed points and, if $x, y \in C_n$ and $f(x) = y$, then there is an $i \leq 1$ such that $\hat{f}_i(x) = y$.
    \item $C_n$ is a cycle. In a similar way, let $\hat{f}_0 (a^n_i) = a^n_{i+1}$ and $\hat{f}_0 (a^n_{i+1}) = a^n_i$ whenever $i \in \mathbb{Z}$ is even, let $\hat{f}_1 (a^n_i) = a^n_{i+1}$, $\hat{f}_0 (a^n_{i+1}) = a^n_i$ when $i \in \mathbb{Z}$ is odd. Define $\hat{f}_2 \restriction C_n,\hat{f}_3 \restriction C_n$ to be any involutions on $C_n$. As in the previous case, it is clear that each $\hat{f}_i$ is an involution without fixed points and, if $x, y \in C_n$ and $f(x) = y$, then there is an $i \leq 1$ such that $f_i(x) = y$.
    \item $C_n$ is a semi cycle. Let $\hat{f}_0 (a^n_i) = a^n_{i+1}$ and $\hat{f}_0 (a^n_{i+1}) = a^n_i$ whenever $i \in \omega$ is even. Let $\hat{f}_1 (a^n_i) = a^n_{i+1}$ and $\hat{f}_1 (a^n_{i+1}) = a^n_i$ if $i \equiv 1 \mod{4}$ and for the undefined (infinitely many) $a^n_i$, define $\hat{f}_1$ in such a way that it is still an involution. Let $\hat{f}_2 (a^n_i) = a^n_{i+1}$ and $\hat{f}_2 (a^n_{i+1}) = a^n_i$ if $i \equiv 3 \mod{4}$ and for the undefined (infinitely many) $a^n_i$, define $\hat{f}_2$ in such a way that it is still an involution. Define $\hat{f}_3 \restriction C_n$ to be any random involution on $C_n$. It is clear that each $\hat{f}_i$ is an involution without fixed points and, if $x, y \in C_n$ and $f(x) = y$, then there is an $i \leq 2$ such that $\hat{f}_i(x) = y$.
\end{itemize}

It follows that each $\hat{f}_i \in \square_\text{involution}$ and that for every $x,y \in \omega$, if $f(x) = y$, then there is an $i \leq 2$ such that $\hat{f}_i(x) = y$. Therefore, if $A\in \Baire$ and each $f_i[A] \cap A$ is finite, then necessarily $f[A] \cap A$ must be finite too.

\medskip
\underline{\textbf{Case 2.}} There is an odd number of $C_n$ such that $|C_n|$ is odd and the rest of the $C_i$ have an even amount of elements. 

As in the previous case, it is possible to modify $f$ in finite many values so that there is a single $C_n$ such that $|C_n|$ is odd, so, without loss of generality $C_0$ is the only one with an odd number of elements. Let $D = \bigcup_{k \in \omega} C_{2k}$ and let $E = C_0 \cup \bigcup_{k \in \omega} C_{2k + 1}$. Clearly $D$ and $E$ are infinite sets. We will define $\hat{f}_0,\hat{f}_1,\hat{f}_2,\hat{f}_3$ using the following rules:
\begin{itemize}
    \item $\hat{f}_0 \restriction E,\hat{f}_1 \restriction E$ are random involutions on $E$, and $\hat{f}_2 \restriction D,\hat{f}_3 \restriction D$ are random involutions on $D$,
    \item for $n \neq 0$ and even, $\hat{f}_0 \restriction C_n,\hat{f}_1 \restriction C_n$ are defined exactly as how they were defined in \textbf{Case 1},
    \item for odd $n$, $\hat{f}_2 \restriction C_n,\hat{f}_3 \restriction C_n$ are defined as how $\hat{f}_0 \restriction C_n,\hat{f}_1 \restriction C_n$ were defined in \textbf{Case 1}.
\end{itemize}

After this, we have that each $\hat{f}_i \in \square_\text{involution}$ and that for every $x,y \in \omega \setminus C_0$, if $f(x) = y$, then there is an $i \leq 3$ such that $\hat{f}_i(x) = y$. Then, if $A\in \Baire$ and each $f_i[A] \cap A$ is finite, we have that $f[A] \cap A$ is finite.

\medskip
\underline{\textbf{Case 3.}} There are infinitely many $C_n$ such that $|C_n|$ is odd. 

For $C_n$ where $C_n$ is either infinite or has an even amount of points, the functions $\hat{f}_i \restriction C_n$ are defined exactly as in \textbf{Case 1.} For each $C_n$ such that it has an odd number of elements, let $D_n = C_n \setminus \{ a^n_{k_n}\}$, $E_n = C_n \setminus \{ a^n_{0}\}$ and let $F_n = \{ a^n_{0},a^n_{k_n}\}$ and let $D=\{ a^n_{k_n} : C_n \text{ has an odd amount of elements} \}$, $E = \{ a^n_{0} : C_n \text{ has an odd amount of elements} \}$ and $F =\{ a^n_{i} : C_n \text{ has an odd amount of elements and } 0<i<k_n\}$. Clearly, $D,E,F$ are infinite subsets of $\bigcup \{ C_n: C_n \text{ has an odd amount of elements} \}$ disjoint with $D_n,E_n$ and $F_n$, respectively. For the rest of the proof, we will only consider $n$ such that $C_n$ has an odd number of elements.
\begin{itemize}
    \item First, we will define $\hat{f}_0 (a^n_i) = a^n_{i+1}$ and $\hat{f}_0 (a^n_{i+1}) = a^n_i$ whenever $0 \leq i < k_n$ is even. Clearly, this defines an involution on $\bigcup D_n$. Define $\hat{f}_0 \restriction D$ to be any involution on $D$.
    \item Then, define $\hat{f}_1 (a^n_i) = a^n_{i+1}$ and $\hat{f}_1 (a^n_{i+1}) = a^n_i$ whenever $0 < i \leq k_n$ is odd. This defines an involution on $\bigcup E_n$. Define $\hat{f}_1 \restriction E$ to be any involution on $E$.
    \item Finally, define $\hat{f}_2 (a^n_{k_n}) = a^n_{0}$ and $\hat{f}_2 (a^n_{0}) = a^n_{k_n}$. This defines an involution on $\bigcup F_n$. Define $\hat{f}_2 \restriction F$ to be any involution on $F$. Let $\hat{f}_3$ be any involution in $\bigcup C_n$.  
\end{itemize}
As in \textbf{Case 1}, it follows that each $\hat{f}_i \in \square_\text{involution}$ and that for every $x,y \in \omega$, if $f(x) = y$, then there is an $i \leq 2$ such that $\hat{f}_i(x) = y$. So, if $A\in \Baire$ and each $f_i[A] \cap A$ is finite, then $f[A] \cap A$ is finite.
\end{proof}

This lemma will be used to prove the first main result of this section.

\begin{theorem}  \label{deltasandrosareequal} 
    $\Delta_\text{1-1} = \Delta_\text{involution}$ and $\ros_\text{1-1} = \ros_\text{involution}$.
\end{theorem}
\begin{proof}
    We only need to prove that $\Delta_\text{1-1} \geq \Delta_\text{involution}$ and that $\ros_\text{1-1} \leq \ros_\text{involution}$.
    
    ($\Delta_\text{1-1} \geq \Delta_\text{involution}$): Let $\{ f_\alpha : \alpha \in \Delta_\text{1-1} \}$ be a witness for $\Delta_\text{1-1}$. Using the previous lemma, for each $\alpha \in \Delta_\text{1-1}$, it is possible to find $\hat{f}^\alpha_0,\hat{f}^\alpha_1,\hat{f}^\alpha_2,\hat{f}^\alpha_3$ such that, for all $A \in \Baire$, if $f_\alpha[A]\cap A$ is infinite, then there is an $i \in \{0,1,2,3\}$ such that $\hat{f}^\alpha_i[A]\cap A$ is infinite. It follows that $\{ \hat{f}^\alpha_i : \alpha \in \Delta_\text{1-1}, i\in\{0,1,2,3\} \}$ must be a witness for $\Delta_\text{involution}$.

    ($\ros_\text{1-1} \leq \ros_\text{involution}$): Let $\kappa < \ros_\text{1-1}$ and let $\{ A_\alpha : \alpha \in \kappa \} \subseteq \Baire$, we will show that this family is not a witness for $\ros_\text{involution}$. Since $\kappa < \ros_\text{1-1}$, there is a function $f \in \square_{1-1}$ such that, for all $\alpha \in \kappa$, $f[A_\alpha] \cap A_\alpha$ is infinite. Then, using the previous lemma, there are $\hat{f}_0,\hat{f}_1,\hat{f}_2,\hat{f}_3 \in \square_\text{involution}$ such that for every $\alpha \in \kappa$ there is an $i_\alpha \in \{0,1,2,3\}$ such that $\hat{f}_{i_\alpha}[A_\alpha] \cap A_\alpha$ is infinite. For every $\alpha \in \kappa$, one can easily construct an interval partition $\II_\alpha$, such that, for every $I \in \II_\alpha$, $\hat{f}_{i_\alpha}[I\cap A] \cap (I \cap A)$ has at least 2 elements of the form $n, \hat{f}_{i_\alpha} (n)$. Since $\kappa < \ros_\text{1-1} \leq \ros_\text{Fin-1} \leq \dom$, by Theorem \ref{blassdominating}, there must be an interval partition $\JJ = \{ J_n : n \in \omega \}$ such that, for all $\alpha \in \kappa$, the set $D_\alpha = \{ n \in \omega : J_n \text{ contains an interval from }\II_\alpha\}$ is infinite. Additionally, we may assume that each $|J_n|$ is odd. Since $\kappa < \ros_\text{1-1} \leq \ros_\text{Fin-1} \leq \mathfrak{r}$, there is a partition of $\omega = P_0 \cup P_1 \cup P_2 \cup P_3$ such that, for all $\alpha \in \kappa$ and all $i \in \{0,1,2,3\}$, $P_i \cap D_\alpha$ is infinite. Let
$$
D = \{ n \in \omega : n, \hat{f}_i(n) \in J_k \text{ for }k \in P_i, i \in \omega \}.
$$
Note that $\omega \setminus D$ is infinite since each $J_k$ has an odd number of elements. Observe that, for all $i \in \omega$ and every $k\in P_i$, $n \in D$ if and only if $f_i(n) \in D$ and therefore the function $\hat{f}' = \bigcup_{i \in \{0,1,2,3\}} \bigcup_{k \in P_i} \hat{f}_i \restriction (J_k \cap D)$ is an involution on $D$. Let $\hat{f} \in \square_\text{involution}$ be any function that extends $\hat{f}'$. We will see that, for every $\alpha \in \kappa$, $\hat{f}[A_\alpha] \cap A_\alpha$ is infinite: Let $\alpha \in \kappa$, so $P_{i_\alpha} \cap D_\alpha$ is infinite. For every $k \in P_{i_\alpha} \cap D_\alpha$ we have that $J_k$ contains an interval from $\II_\alpha$, thus $\hat{f}_{i_\alpha}[J_k \cap A] \cap (J_k \cap A)$ has at least 2 elements of the form $n, \hat{f}_{i_\alpha} (n)$. Finally, observe that $n \in D$, so $\hat{f}_{i_\alpha}(n) = \hat{f}(n)$ and therefore $\hat{f}[J_k \cap A] \cap (J_k \cap A) \neq \emptyset$. This implies that $\hat{f}[A]\cap A$ is infinite, which is what we wanted to prove.
\end{proof}

Clearly, the previous equality holds for any property $\varphi$ in between being one to one and being an involution. In particular, $\Delta_\text{bijective} = \Delta_\text{1-1}$ and $\ros_\text{bijective} = \ros_\text{1-1}$.

We will now relate these cardinal characteristics to the ones appearing in \cite{Banakh} and \cite{BanakhProtasov}. In those works, the authors considered a slight modification of $\Delta_\text{involution}$ and $\ros_\text{involution}$, where a single fixed point was allowed. Clearly an involution with a single fixed point can be modified in a finite way to be a bijection: For instance, if $f$ is an involution such that $n \in \omega$ is its fixed point, and if $m \neq n$, then setting $\hat{f}(n) = m, \hat{f}(m) = f(n)$ and leaving the rest of $\hat{f}$ the same as $f$ yields a bijection without fixed points such that, for every $A \in \Baire, f[A] \cap A$ is infinite if and only if $\hat{f}[A] \cap A$ is infinite. As a consequence, $\Delta_\text{involution} \leq \Delta \leq \Delta_\text{bijection}$ and $\ros_\text{involution} \geq \widehat{\Delta} \geq \ros_\text{bijection}$, giving the following as a consequence.
\begin{corollary}\label{deltabanakhdelta}
    $\Delta = \Delta_\text{1-1}$ and $\widehat\Delta = \ros_\text{1-1}$.
\end{corollary}
\begin{proof}
    It follows from Theorem \ref{deltasandrosareequal} and the remark above.
\end{proof}

We can now apply the results from \cite{Banakh} to $\Delta_\text{1-1}$ to get the following.

\begin{theorem}
    $\max\{\unb, \spl, \cov{\NN} \} \leq \Delta_\text{1-1} \leq \non{\MM}$ and $\min\{\dom,\mathfrak{r},\non{\NN} \} \geq \ros_\text{1-1} \geq \cov{\MM}$. \label{banakhbounds}
\end{theorem}
\begin{proof}
    This is a direct consequence of Corollary \ref{deltabanakhdelta} and Theorem \ref{maintheoremBanakh} (see \cite[Theorems 3.2 and 7.1]{Banakh}). Alternatively one can use Proposition \ref{slessdelta}, Theorem \ref{deltafin1charact}, the fact that $\ros_\text{Fin-1} = \min\{\dom,\mathfrak{r}\}$ (\cite[Theorem 27]{ArturoKoszmider}), Proposition \ref{deltaincreasing} and Theorem \ref{eg_bounds}.
\end{proof}

In the final section of this work, we will show the consistency of the strict ineqalities $\max\{\unb, \spl, \cov{\NN} \} < \Delta_\text{1-1}$ and $\min\{\dom,\mathfrak{r},\non{\NN} \} > \ros_\text{1-1}$.
We do not know the answer to the following question.

\begin{question}\label{delta1-1}
Is it true that $\Delta_\text{1-1}=\non{\MM}$ and $\ros_{\text{1-1}}=\cov{\MM}$?
\end{question}

\noindent The positive answer would provide negative answers to Problems $3.8$, $8.11.(2)$ and $8.11.(3)$, appearing in \cite{Banakh}.

To conclude this section, we would like to take a look at the case where the functions are strictly increasing ($\varphi=$increasing). Since all these functions are injective, then it follows that $\Delta_\text{1-1} \leq \Delta_\text{increasing}$ and $\ros_\text{1-1} \geq \ros_\text{increasing}$. The following is also easy to see.

\begin{proposition} \label{deltaincreasing}
    $\Delta_\text{increasing}\leq \non{\MM}$ and $\ros_{\text{increasing}}\geq\cov{\MM}$.
\end{proposition}
\begin{proof}
    The proof is similar to the proof of Proposition \ref{splmixnonm}: It follows from the fact that $\square_\text{increasing}$ is Polish and the fact that, for every $A \in \Baire$, $\{ f \in \square_\text{increasing} : f[A] \cap A \text{ is finite} \}$ is meager.
\end{proof}

There is nothing else we know about $\Delta_\text{increasing}$ and $\ros_\text{increasing}$, leaving us with the following natural question.

\begin{question}
    What are the values of $\Delta_\text{increasing}$ and $\ros_\text{increasing}$?
\end{question}

\section{Eventually different and infinitely equal reals.}

Recall that, given $f,g \in \baire$, $f$ is \emph{infinitely equal to} $g$, denoted by $f =^\infty g$, if $f \cap g$ is infinite. This naturally yields a relational system $\mathcal{E} = \langle \baire, \baire, =^\infty \rangle$ and, by a well-known theorem of Bartoszy\'nski \cite{BartSlalom} and Miller \cite{Miller2}, $\dom(\mathcal{E}) = \non{\MM}$ and $\unb(\mathcal{E}) = \cov{\MM}$ (see also \cite[Section 2.4]{Bartoszynski}\label{Bartoszynski}). We will make use of the following bounded variation of this relation:

\begin{definition}
Given a strictly increasing $g\in\baire$, the following is the $g$\emph{-bounded version of the infinitely equal number}.
$$
\mathfrak{e}_g= \min \big\{ \big|\FF\big|\colon \FF \sub\prod_{n\in\omega}g(n) \text{ and }\: \forall h\in\prod_{n\in\omega} g(n) \: \exists f\in\FF \: (f =^\infty h) \big\}.
$$
and its dual, the $g$\emph{-bounded version of the eventually different number}:
$$
\widehat{\mathfrak{e}_g}= \min \big\{ \big|\FF\big|\colon \FF \sub\prod_{n\in\omega}g(n)  \text{ and }\: \forall h\in\prod_{n\in\omega} g(n) \: \exists f\in\FF \: (f \cap h \text{ is finite}) \big\}.
$$
\end{definition}

These invariants have been widely studied in the literature, even in a much more general setting (\cite[Section 2]{Miller}, \cite{Yorioka}). Clearly $\mathfrak{e}_g \leq \non{\MM}$ and $\widehat{\mathfrak{e}_g} \geq \cov{\MM}$. Consistently, these invariants can have many different values (see \cite{MAnySimple} and \cite{KamoOsuga}). Another well-known fact about these invariants  is their relation to the ideal of sets of measure zero (see e.g. \cite[Lemma 2]{KamoOsuga} and \cite[Theorem 3.21]{Yorioka}) which we will prove for the sake of completeness of the paper:

\begin{proposition}\label{eg_null}
$\cov{\NN}\leq\mathfrak{e}_g$ and $\non{\NN}\geq\widehat{\mathfrak{e}_g}$ for every $g\in\baire$ such that $\sum_{n\in\omega}1/g(n)<\infty$.
\end{proposition}
\begin{proof}
Let $\lambda_g$ be the standard product probability measure on 
$\prod_{n\in\omega}g(n)$ and let $f\in\baire$. For every $k\in\omega$ we have $\lambda_g\big(\{h\in\prod_{n\in\omega}g(n)\colon h(k)=f(k)\}\big)=1/g(k)$. Therefore, by the Borel--Cantelli lemma, we get that:
$$\lambda_g\big(\big\{h\in\prod_{n\in\omega}g(n) : ( h =^\infty f) \big\}\big) = 0. $$
Thus, a family witnessing $\mathfrak{e}_g$ is also a witness for $\cov{\NN}$. The second inequality follows dually.
\end{proof}

As a conclusion, we have $\cov{\NN} \leq \mathfrak{e}_g \leq \non{\MM}$ and $\cov{\MM}\leq \widehat{\mathfrak{e}_g} \leq \non{\NN}$. Another easy thing to see is that $\mathfrak{e}_g\leq \mathfrak{e}_f$ whenever $g\leq^* f$, as if $\FF$ is a witness for $\mathfrak{e}_f$ then $\GG=\{\min(h, g) : h\in\FF\}$ is a witness for $\mathfrak{e}_g$. By an analogous reason, we have that $\widehat{\mathfrak{e}_g} \geq \widehat{\mathfrak{e}_f}$ whenever $g\leq^* f$. By the result of Miller \cite[Theorem 2.3]{Miller}, the infimum of all  $\widehat{\mathfrak{e}_g}$ is equal to $\non{\mathcal{SN}}$, where $\mathcal{SN}$ is the ideal of the sets of strong measure zero on $2^\omega$.

Our next goal will be to prove $\Delta_{1-1}\geq \mathfrak{e}_g$ and $\ros_{1-1}\leq \widehat{\mathfrak{e}_g}$ for every increasing $g\in\baire$, which, by Proposition \ref{eg_null} and Theorem \ref{deltasandrosareequal}, improve the bounds $\Delta_{1-1}\geq\cov{\NN}$ and $\ros_{1-1}\leq\non{\NN}$ obtained in \cite[Lemma 3.7 and Lemma 7.4]{Banakh}.


\begin{theorem}\label{eg_bounds}
$\Delta_{1-1}\geq \mathfrak{e}_g$ and $\ros_{1-1}\leq \widehat{\mathfrak{e}_g}$ for every $g\in\baire$.
\end{theorem}

\begin{proof}
Let $g\in\baire$ be increasing. Before we begin the proof, we will build a couple of interval partitions, functions and some claims that will help us with the proof. First, pick an interval partition $\{I_n\colon n\in\omega\}$ of $\omega$ such that for every $n\in\omega$ we have:
$$
|I_n| > 2\cdot\sum_{j<n}\prod_{i\in I_j}g(i).
$$
Let us define a function $F$ by setting $F(n)=\prod_{i\in I_n} g(i)$. By the condition for the value of $|I_n|$ we have 
$$
|I_n| > 2\cdot\sum_{j<n} F(j).
$$
Let $\PP=\{J_n\colon n\in\omega\}$ be an interval partition of $\omega$ such that $|J_n|=F(n)$ for every $n\in\omega$.
 
For every $f\in\square_\text{1-1}$ let us define a function $S_f:\omega\to[\omega]^{<\omega}$ in the following way:
$$
S_f(n)=\Bigg(f\Big[\bigcup_{i<n}J_i\Big] \cup f^{-1}\Big[\bigcup_{i<n}J_i\Big]\Bigg)\cap J_n \subseteq J_n.
$$

Note that, by the injectivity of $f$ we have
$$
|S_f(n)|\leq 2\cdot\sum_{i<n}|J_i| \leq 2\cdot\sum_{j<n} F(j) < |I_n|.
$$
Since $|J_n| = F(n)$ there must be a bijection $b_n : J_n \rightarrow \prod_{i\in I_n} g(i)$, so $b_n[S_f(n)]$ is a collection of functions in $\prod_{i\in I_n} g(i)$. We will need the following two observations.
\begin{claim}
Let $h \in \prod_{n\in\omega}J_n$. Then, for all $f\in\square_\text{1-1}$ such that for almost all $n \in \omega$, $h(n)\notin S_f(n)$, we have that $f[A]\cap A$ is finite, for $A = h[\omega]$.
\end{claim}
\begin{proof}[Proof of the claim.] Assume that $n \in \omega$ is such that $h(n) \in f[A]$. Then, there is an $m \neq n$ such that $h(m) \in A$ such that $f(h(m)) = h(n)$.
\begin{itemize}
    \item if $m < n$, then $h(n) \in f[\bigcup_{i < n} J_i]$ and therefore $h(n) \in S_f(n)$,
    \item if $n < m$ then $h(m) \in f^{-1}[\bigcup_{i < m} J_i]$ and therefore $h(m) \in S_f(m)$.
\end{itemize}
Since $h(i) \in S_f(i)$ only for finitely many $i \in \omega$, then the conclusion follows.    
\end{proof}

\begin{claim}
    For every $f \in \square_\text{1-1}$ there is a function $\ell_f \in \prod_{n \in \omega} g(n)$ such that, for all $n \in \omega$ and all $h \in b_n[S_f(n)],$ we have that $(\ell_f \restriction I_n)\cap h \neq \emptyset$. 
\end{claim}
\begin{proof}[Proof of the claim.]
    This is a consequence of the fact that there are more points in $I_n$ than in $b_n[S_f(n)]$, so there is a function $\ell_f \restriction I_n \in \prod_{i \in I_n} g(i)$ such that $(\ell_f \restriction I_n) \cap h \neq \emptyset$ for all $h \in b_n[S_f(n)]$. Clearly, $f = \bigcup_{n \in \omega} \ell_f \restriction I_n$ is the function we are looking for.
\end{proof}

We are ready to prove our results. First, we will prove that $\Delta_{1-1}\geq \mathfrak{e}_g$: Let $\FF=\{f_{\alpha}\colon\alpha<\Delta_{1-1}\} \subseteq \square_{1-1}$ witnessing $\Delta_{1-1}$. We will see that $\{ \ell_{f_\alpha} : \alpha \in \kappa \}$ witnesses $\mathfrak{e}_g$: Let $h \in \prod_{n \in \omega} g(n)$ and let $A = \{ b_n^{-1}(h \restriction I_n) : n \in \omega \}$. Clearly $A \in \Baire$ so there must be an $\alpha \in \kappa$ such that $f_\alpha [A] \cap A$ is infinite. Therefore, by the first claim, there must be an infinite amount of $n\in\omega$ such that $b_n^{-1}(h \restriction I_n) \in S_{f_\alpha}$. For such $n$, by the second claim, we have that $(\ell_{f_\alpha} \restriction I_n) \cap b_n ( b_n^{-1}(h \restriction I_n)) \neq \emptyset$ and therefore $f_\alpha \cap h$ is infinite.

Finally, we will prove that $\ros_{1-1}\leq \hat{\mathfrak{e}_g}$. Let $\kappa < \ros_{1-1}$, and let $\{h_\alpha : \alpha \in \kappa\} \subseteq \prod_{n \in \omega} g(n)$. We will find a function $\ell \in \prod_{n \in \omega} g(n)$ such that $\ell \cap h_\alpha$ is infinite for all $\alpha \in \kappa$, so $\{h_\alpha : \alpha \in \kappa\}$ cannot be a witness of $\hat{\mathfrak{e}_g}$. Let $A_\alpha = \{ b_n^{-1}(h_\alpha \restriction I_n) : n \in \omega \} \in \Baire$. Then, since $\kappa < \ros_{1-1}$ there must be an $f \in \square_{\text{1-1}}$ such that $f[A_\alpha] \cap A_\alpha$ is infinite for all $\alpha \in \kappa$. By the first claim, for all $\alpha \in \kappa$ there is an infinite amount of $n \in \omega$ such that $b_n^{-1}(h_\alpha \restriction I_n) \in S_{f}$. Using the second claim, we conclude that $\ell_f \cap h_\alpha$ is infinite for every $\alpha \in \kappa$, which is what we were looking for.
\end{proof}

For our final result of this section, we will need to recall the following widely-known result, which we prove for the sake of completeness of the paper.

\begin{theorem}\label{egscofaddmeager}
    $\max\{\dom, \sup_{g\in\baire}{\mathfrak{e}_g}\}=\cof{\MM}$ and $\min\{\unb, \min_{g\in\baire}\widehat{\mathfrak{e}_g}\}=\add{\MM}$.
\end{theorem}
\begin{proof}
By Miller's theorem \cite[Theorem 1.2]{Miller}, $\max\{\dom, \non{\MM}\}=\cof{\MM}$ and $\min\{\unb, \cov{\MM}\}=\add{\MM}$ (see also \cite[Theorem 5.6]{Blass}). Thus, it is enough to prove that $\max\{\dom, \sup({\mathfrak{e}_g})\}\geq\non{\MM}$ and $\min\{\unb, \min(\widehat{\mathfrak{e}_g})\}\leq\cov{\MM}$. We are going to use the characterizations of $\non{\MM}$ and $\cov{\MM}$ as $\dom(\EE)$ and $\unb(\EE)$ mentioned at the beginning of this section.

($\max\{\dom, \sup({\mathfrak{e}_g})\}\geq\dom(\EE)$). Let $\{f_{\alpha}\colon\alpha<\dom\}$ be a dominating family in $\baire$ of strictly increasing functions. For every $\alpha<\dom$, pick a witness $\FF_{\alpha}\sub\prod_{n\in\omega} f_{\alpha}(n)$ for $\mathfrak{e}_{f_{\alpha}}$. Clearly $\bigcup_{\alpha<\dom}\FF_{\alpha}$ is a witness for $\dom(\EE)$ of the required size. 

$(\min\{\unb, \min(\widehat{\mathfrak{e}_g})\}\leq\unb(\EE))$. Let $\kappa<\min\{\unb, \mathfrak{ros_{\text{1-1}}}\}$ and let $\{f_{\alpha}\colon\alpha<\kappa\} \in \baire$. Since $\kappa<\unb$, there exists $g\in\baire$ strictly increasing such that $f_{\alpha}\leq^* g$ for every $\alpha<\kappa$, and so for each $\alpha<\kappa$ there is a finite modification $f_{\alpha}'$ of $f_{\alpha}$ satisfying $f_{\alpha}'\in\prod_{n\in\omega} g(n)$. As $\kappa< \widehat{\mathfrak{e}_g}$, there is $h\in\prod_{n\in\omega} g(n)$ such that for every $\alpha<\kappa, f'_{\alpha} =^\infty h$ for infinitely many $n\in\omega$, thus $f_{\alpha}=^\infty h$ and therefore $\{f_{\alpha}\colon\alpha<\kappa\}$ is not an $\EE$-unbounded family.
\end{proof}

As a corollary, we get the following result connecting $\Delta_{\text{1-1}}$ with $\cof{\MM}$, and $\mathfrak{ros_{\text{1-1}}}$ with $\add{\MM}$, respectively.

\begin{corollary}\label{deltaroscofaddmeager}
$\max\{\dom, \Delta_{\text{1-1}}\}=\cof{\MM}$ and $\min\{\unb, \mathfrak{ros_{\text{1-1}}}\}=\add{\MM}$.
\end{corollary}
\begin{proof}
It follows from Theorems \ref{banakhbounds}, \ref{eg_bounds} and \ref{egscofaddmeager}.
\end{proof}

For our final result of this section, we will need some terminology. Recall that an \emph{ideal (on $\omega$)} is a proper collection $\mathcal{I} \subsetneq \mathcal{P}(\omega)$ closed under finite unions and subsets. Additionally, we will require that our ideals contains all singletons. An ideal $\mathcal{I}$ is \emph{tall} (or \emph{dense}) if for every $X \in \Baire$ there is an infinite $I \in \mathcal{I}$ such that $I \subseteq X$. The following are two of the most common cardinal invariants related to tall ideals on $\omega$, originally introduced in \cite[Definition 1.1]{HHH}.

\begin{definition}\label{edmeasure}
    Let $\mathcal{I}$ be a tall ideal on $\omega$, the \emph{covering number} and the \emph{uniformity number} are given by
    $$
\covstar{\II} = \min \big\{ |\AAA| \colon \AAA\sub\II  \text{ and }\: \forall X\in\Baire \: \exists A\in\AAA \:(|A\cap X|=\aleph_0) \big\},
    $$
    $$
\nonstar{\II} = \min \big\{ |\mathcal{X}| \colon \mathcal{X} \subseteq \Baire  \text{ and }\: \forall A \in \mathcal{I} \: \exists X \in \mathcal{X} \: (|A\cap X| < \aleph_0)\big\}.
    $$
\end{definition}

Given an interval partition $\JJ =\{ J_n \colon n \in \omega \}$, and a sequence of measures $\bar{\mu} = \{\mu_n : n \in \omega \}$ such that $\mu_n$ is a measure on $J_n$, its value on singletons is at most $1$ and $\mu_n (J_n) \geq n$, the ideal $\mathcal{ED}_{\Bar{\mu}}$ is defined in the following way:
$$
\EE\DD_{\bar{\mu}}=\big\{A\sub\omega \colon \exists k\in\omega \: \forall n\in\omega \: \big(\mu_n(J_n\cap A)\leq k\big)\big\}.
$$
One can easily see that all the ideals of the form $\EE\DD_{\bar{\mu}}$ are tall, $F_\sigma$ and both $\covstar{\EE\DD_{\bar{\mu}}}$ and $\nonstar{\EE\DD_{\bar{\mu}}}$ are uncountable. In the literature, the ideal $\EE\DD_\text{Fin}$ is defined as an ideal on $\{ \langle m, n \rangle : m \leq n \}$ generated by the graphs of functions of $\baire$. Clearly ideal $\EE\DD_\text{Fin}$ is of the form $\EE\DD_{\bar{\mu}}$, where $J_n$ is a set of size $n$ and $\mu_n$ is the counting measure on $J_n$ for every $n\in\omega$. In \cite[Proposition 3.6]{PairSplitting}, the authors prove the following characterization of $\non{\MM}$ and $\cov{\MM}$.

\begin{theorem}[Hru\v{s}\'{a}k, Meza, Minami]
    $\cov{\MM} = \min \{ \dom, \nonstar{\EE\DD_\text{Fin}}\}$ and $\non{\MM} = \max \{ \unb, \covstar{\EE\DD_\text{Fin}}\}$.
\end{theorem}

Since $\ros_\text{1-1} \leq \mathfrak{d}$ and $\Delta_\text{1-1} \geq \unb$, the question whether $\ros_\text{1-1} = \cov{\MM}$ and $\Delta_\text{1-1}$ reduces to find out if $\ros_\text{1-1} \leq \nonstar{\EE\DD_\text{Fin}}$ and $\Delta_\text{1-1} \geq \covstar{\EE\DD_\text{Fin}}$. We were not able to prove it for $\EE\DD_\text{Fin}$, but we were able to prove those inequalities for some other ideal $\EE\DD_{\bar{\mu}}$.

\begin{theorem}\label{fsigma_filter}
There is an ideal of the form $\EE\DD_{\bar{\mu}}$ such that $\Delta_\text{1-1}\geq\covstar{\EE\DD_{\bar{\mu}}}$ and $\ros_\text{1-1}\leq\nonstar{\EE\DD_{\bar{\mu}}}$. 
\end{theorem}
\begin{proof}

Fix an interval partition $\{J_n : n\in\omega\}$ of $\omega$ such that $|J_0|=1$ and for every $n\in\omega$:

$$|J_{n+1}| = 2(n+1)\cdot\sum_{k\leq n}|J_k|.$$

For every $n\in\omega$, define a measure $\mu_n$ on $J_n$ by setting $\mu_0(J_0)=0$ and
$$
\mu_{n+1}(\{i\}) = \frac{1}{\sum_{k\leq n}|J_k|} \textrm{ for any } n\in\omega \textrm{ and } i\in J_n.   
$$
We will show that, if $\bar{\mu} = \langle \mu_n : n \in \omega \rangle$, then $\EE\DD_{\bar{\mu}}$ is the ideal we are looking for. First, we will need the following claims:

\begin{claim}
    For every $f \in \square_\text{1-1}$, the set $B_f \in \EE\DD_{\bar{\mu}}$, where
    $$ B_f = \bigcup_{n\in\omega}\Big(J_n\cap\bigcup_{k<n}\big( 
f[J_k]\cup f^{-1}[J_k] \big) \Big).$$
\end{claim}
\begin{proof}[Proof of the Claim.]This follows from the following calculations:
$$
\mu_n (B_f \cap J_n) = \mu_n (J_n\cap\bigcup_{k<n}\big( f[J_k]\cup f^{-1}[J_k] )) = 
$$
$$
= \frac{|J_n\cap\bigcup_{k<n}\big( f[J_k]\cup f^{-1}[J_k] )|}{\sum_{k<n} |J_k|}\leq\frac{2 \cdot \sum_{k<n} |J_k|}{\sum_{k<n} |J_k|} \leq 2.
$$
\end{proof}
\begin{claim}
    If $X \subseteq \omega$ is such that $|X \cap J_n|\leq 1$ and $X \cap B_f$ is finite, then $f[X] \cap X$ is finite.
\end{claim}
\begin{proof}
    This proof is similar to the first claim of Theorem \ref{eg_bounds}: We may assume that $|X \cap J_n| = 1$. Let $\{ x_k : k \in \omega \}$ be an enumeration of $X$ such that $x_k \in J_k$. Let $N \in \omega$ be such that $X \cap B_f \subseteq N$. If $m,n \in \omega$ are such that $x_m,x_n > N$, then:
    \begin{itemize}
        \item if $x_m < x_n$ then $f(x_m) \in f[\bigcup_{i<n} J_i]$. Since $x_n \in J_n \setminus B_f$, then $x_n \neq f(x_m)$,
        \item if $x_m > x_n$ then $f^{-1}(x_n)  \in f^{-1}[\bigcup_{i<m} J_i]$. Since $x_m \in J_m \setminus B_f$, then $x_n \neq f(x_m)$.
    \end{itemize}
    Therefore, for any $x_m,x_n > N$, $x_n \neq f(x_m)$ so $f[X] \cap X$ is finite. 
\end{proof}

We will now continue with the proof of the main Proposition.

($\Delta_\text{1-1}\geq\covstar{\EE\DD_{\bar{\mu}}}$). Let $\{f_\alpha\colon\alpha<\kappa\} \sub\square_\text{1-1}$, where $\kappa<\covstar{\EE\DD_{(\mu_n)}}$. Then $\{ B_{f_\alpha} : \alpha \in \kappa \} \subseteq \EE\DD_{\bar{\mu}}$ so there must exist an $X \in \Baire$ such that $X \cap B_{f_\alpha}$ is finite for each $\alpha < \kappa$. By shrinking $X$, we may assume that $|X \cap J_n| \leq 1$. Then, by the second Claim, for every $\alpha \in \kappa$, $f_\alpha [X] \cap X$ is finite.

($\ros_\text{1-1}\leq\nonstar{\EE\DD_{\bar{\mu}}}$). Let $\{ X_\alpha :\alpha \in \nonstar{\EE\DD_{\bar{\mu}}}\}$ be a witness for $\nonstar{\EE\DD_{\bar{\mu}}}$. By shrinking each $X_\alpha$, we may assume that, for each $\alpha \in \nonstar{\EE\DD_{\bar{\mu}}}$ and each $n\in \omega$, $|X_\alpha \cap J_n| \leq 1$. We will show that $\{ X_\alpha :\alpha \in \nonstar{\EE\DD_{\bar{\mu}}}\}$ is a witness for $\ros_\text{1-1}$: Let $f \in \square_\text{1-1}$. Then $B_f \in \EE\DD_{\bar{\mu}}$, so there must be an $\alpha$ such that $B_f \cap X_\alpha$ is finite. Then, by the second claim, $f[X_\alpha] \cap X_\alpha$ is finite.
\end{proof}

To finish this section, we would like to point out that it is easy to see that, for all sequences of measures $\bar{\mu}$ as in Definition \ref{edmeasure},  $\nonstar{\EE\DD_{\text{Fin}}} \leq \nonstar{\EE\DD_{\bar{\mu}}}$ and $\covstar{\EE\DD_{\text{Fin}}} \geq \covstar{\EE\DD_{\bar{\mu}}}$ (a more general version of this statement can be found in \cite[Theorem 3.3]{PairSplitting}), raising the following question.

\begin{question}
    Is it consistent that $\nonstar{\EE\DD_{\text{Fin}}} < \nonstar{\EE\DD_{\bar{\mu}}}$ or $\covstar{\EE\DD_{\text{Fin}}} > \covstar{\EE\DD_{\bar{\mu}}}$ for the sequence $\bar{\mu}$ obtained in Theorem \ref{fsigma_filter}?
\end{question}

\section{Consistency results}

The final section will be dedicated to prove consistency results around all the cardinal invariants that we used in this work. We will start by looking at an equivalent variant of having the Laver property.

\begin{definition}
    A forcing $\mathbb{P}$ has the \emph{Laver property} if for every $\mathbb{P}$-name for a function $\dot{g}$ such that there is $f\in \baire$ such that
    $$
\mathbb{P} \forces{\dot{g} \leq f}
    $$
    then there is a function $S$ such that, for all $n>1, S(n) < |n|$ and 
    $$
\mathbb{P} \forces{\forall n > 1 (\dot{g}(n) \in S(n))}.
    $$
\end{definition}

It is known that the Laver property is preserved under countable support iterations of proper forcings (see \cite[Theorem 6.3.30]{Bartoszynski}). For our next result, we will be interested into looking at the behaviour of $\ros_{1-1}$ on extensions of forcings that have the Laver property. For our cases, it would be enough to look at $\widehat{\mathfrak{e}_g}$. The following implicitly appears in \cite{KadaThesis} (see also \cite{Kada}).

\begin{proposition}\label{ehlaver}
Assume $V$ satisfies CH and that $\mathbb{P}$ is a proper forcing that has the Laver property. Let $g\in\baire$ be strictly increasing. Then,
$$
\mathbb{P} \forces{V \cap \prod_{n \in \omega}g(n) \text{ witnesses both }\mathfrak{e}_g \text{ and }\widehat{\mathfrak{e}_g} }.
$$
In particular, $\mathbb{P} \forces{\mathfrak{e}_g = \widehat{\mathfrak{e}_g} = \aleph_1}.$
\end{proposition}
\begin{proof}
Let $\dot{f}$ be a $\mathbb{P}$-name for a function in $\prod_{n \in \omega}g(n)$. We will show that there are $\hat{f}, \bar{f} \in \prod_{n \in \omega}g(n)$ such that $\mathbb{P} \forces{\hat{f} =^\infty \dot{f} \text{ and }\bar{f}\cap\dot{f} \text{ is finite}}$: First, find $\{ I_n : n \in \omega \}$ an interval partition of $\omega$ such that $|I_n| \geq n$. Use the Laver property to find a function $S$ such that $S(n)$ is a family of functions from $I_n$ to $\omega$ such that $\mathbb{P} \forces{\dot{f} \restriction I_n \in S(n)}$ and $|S(n)| < n$ for $n > 1$. Since $|S(n)| < |I_n|$, there exists a function $\hat{f}_n \in \prod_{j \in I_n}g(j)$ such that $\hat{f}_n \cap h \neq \emptyset$ for every $h \in S(n)$. So if $\hat{f} = \bigcup \hat{f}_n$, then $\hat{f} \in \prod_{n \in \omega}g(n)$ and $\mathbb{P} \forces {\hat{f} =^\infty \dot{f}}$, which finishes the first part of the proof. To find $\bar{f}$, use the Laver property to find a function $S$ such that for all $n>1$, $|S(n)| < n$ and $\mathbb{P}\forces{\dot{f}(n) \in S(n) }$. Since $|S(n)| < n$, there exists a function $\bar{f} \in \prod_{n \in \omega}g(n)$ such that, for each $n>1, \bar{f}(n) \neq S(n)$. Therefore $\mathbb{P} \forces {\bar{f}\cap \dot{f} \text{ is finite}}$, which finishes the proof. 
\end{proof}

For our next results, we will need the following widely known class of forcings.

\begin{definition}
    Let $\mathcal{X}$ be either a filter, or $\Baire$. The \emph{Mathias forcing of} $\mathcal{X}$ is the following:
    $$
\mathbf{M}(\mathcal{X}) = \{ \langle s, A \rangle : s \in [\omega]^{<\omega}, A \in \mathcal{X} \},
    $$
    and the order is defined by $\langle s, A\rangle \leq \langle s', A' \rangle$ if and only if $s' \subseteq s$, $A' \subseteq A$ and $s \setminus s' \subseteq A'$. $\mathbf{M}(\Baire)$ is known simply as the \emph{Mathias forcing} and is denoted by $\mathbf{M}$.
\end{definition}

Notice that, for all filters $\mathcal{F}$, $\mathbf{M}(\mathcal{F})$ is a $\sigma$-centered forcing notion, thus ccc. $\mathbf{M}$ is not ccc, but it is proper and has the Laver property (see for example \cite[Section 7.4.A]{Bartoszynski}).

We are ready to show our first consistency result.

\begin{theorem}\label{ros_mathias}
    It is consistent with ZFC that $\ros_\text{1-1} < \min \{ \dom, \mathfrak{r}, \non{\NN} \}$.
\end{theorem}
\begin{proof}
    We force with $\mathbf{M}_{\omega_2}$, the countable support iteration of length $\omega_2$ of the Mathias forcing, over a model of CH. Since the Mathias forcing is proper and has the Laver property, $\mathbf{M}_{\omega_2}$ is also proper and has the Laver property (by \cite[Theorem 6.3.34]{Bartoszynski}), thus $\mathbf{M}_{\omega_2}\forces{\ros_\text{1-1} \leq \widehat{e}_g = \aleph_1}$ by Proposition \ref{ehlaver}. On the other hand, one can easily show that the generic real for Mathias forcing is reaping and codes a dominating function, and so $\mathbf{M}_{\omega_2}\forces{\aleph_2 = \mathfrak{b} \leq \min \{\mathfrak{r}, \mathfrak{d}\} = 2^{\aleph_0}}$ and $\mathbf{M}_{\omega_2}\forces{\aleph_2 = \mathfrak{s} \leq \non{\NN} = 2^{\aleph_0}}$. Alternatively, one can check the values in the Mathias model in \cite[Model 7.6.11]{Bartoszynski}.
\end{proof}

Let us note the values of $\mathfrak{e}_g$ and $\nonstar{\EE\DD_\text{Fin}}$ in the Mathias model: By Proposition \ref{ehlaver}, we have that, in the extension, $\mathfrak{e}_g = \aleph_1$, and by \cite[Lemma 4.4]{PairSplitting}, it holds that $\nonstar{\EE\DD_\text{Fin}} = \aleph_1$. The value of $\covstar{\EE\DD_\text{Fin}}$ and $\Delta_\text{1-1}$ is $\aleph_2$ in the Mathias model, since they are both bounded below by $\spl = \aleph_2$ (it follows from the results in \cite{PairSplitting} concerning the $\leq_{KB}$-minimality of $\EE\DD_\text{Fin}$ among all $\omega$-hitting $F_\sigma$ ideals).

Our next consistency result is the following.

\begin{theorem}\label{delta_mathias}
It is consistent with ZFC that $\Delta_\text{1-1}>\max \{\unb, \spl, \cov{\NN} \}.$
\end{theorem}
\begin{proof}
We start with a model $V$ of CH, and consider $\mathbf{P}_{\omega_2}$, the finite support iteration of length $\omega_2$ of the forcing $\mathbf{M}(\FF)$ with $\FF = \II^*$, where $\II$ is the $F_\sigma$ ideal obtained in Theorem \ref{fsigma_filter}. Then, we have $\mathbf{P}_{\omega_2} \forces{\aleph_2 \leq \covstar{\II} \leq \Delta_\text{1-1}}$ (cf. \cite[Theorem 4.5]{HHH}). We only have to argue that $\mathbf{P}_{\omega_2} \forces{\aleph_1 = \unb = \spl = \cov{\NN}}$:
\begin{itemize}
\item $\mathbf{M}(\FF)$ for an $F_{\sigma}$ filter $\FF$ preserves well-ordered unbounded families (see either \cite[Theorem 3.1(Case 1)]{Canjar} or \cite[Proposition 5]{Canjar2}), and so $\mathbf{P}_{\omega_2}$ does not add \emph{dominating reals}, i.e. $V\cap \baire$ is unbounded in the extension (see \cite[Lemma 6.5.7]{Bartoszynski}), thus $\mathbf{P}_{\omega_2} \forces{\unb=\aleph_1}$,
\item $\mathbf{M}(\FF)$ is a Souslin forcing, as defined in \cite{Souslin} (since $\mathcal{F}$ is $F_\sigma$), and so the finite support iteration $\mathbf{P}_{\omega_2}$ does not add reaping reals, i.e. $V \cap \Baire$ is a splitting family in the extension (see \cite[Claim 3.6]{Souslin}), thus $\mathbf{P}_{\omega_2} \forces{\spl=\aleph_1}$,
\item $\mathbf{M}(\FF)$ is a $\sigma$-centered forcing notion for every filter $\FF$, and so the finite support iteration $\mathbf{P}_{\omega_2}$ does not add random reals, i.e. $\mathcal{N}\cap V$ covers $2^\omega$ (see \cite[Theorems 6.5.29 and 6.5.30]{Bartoszynski}), so $\mathbf{P}_{\omega_2} \forces{\cov{\NN}=\aleph_1}$.
\end{itemize}
By all of this, it follows that $\mathbf{P}_{\omega_2}\forces{\Delta_\text{1-1}>\max \{\unb, \spl, \cov{\NN} \}}$, so our proof is complete.
\end{proof}

We would like to point out some remarks about the previous model. Since we are forcing with finite support iteration, Cohen reals are always added, so $\cov{\MM} = \aleph_2$ in the extension, and therefore $\mathbf{P}_{\omega_2} \forces{\ros_\text{1-1} = \widehat{\mathfrak{e}_g} = \nonstar{\EE\DD_{\bar{\mu}}} = \aleph_2}$. The fact that $\mathbf{P}_{\omega_2} \forces{\mathfrak{e}_g = \aleph_1}$ follows from a general preservation theorem (see \cite[Theorem 4.16 and Example 4.17.(2)]{Yorioka}).

In the previous theorem, instead of considering the ideal obtained in Theorem \ref{fsigma_filter}, one could take the ideal $\EE\DD_\text{fin}$ and everything would work mostly the same. The only potential difference would be the value of $\Delta_\text{1-1}$, raising the following natural question.

\begin{question}
    Assume $V$ satisfies CH, and force with the countable support iteration of length $\omega_2$ of the forcing $\mathbf{M}(\EE\DD_\text{fin}^*)$. What is the value of $\Delta_\text{1-1}$ in the extension?
\end{question}

We would like to finish this work by showing the consistency of $\Delta_\text{1-1} > \covstar{\EE\DD_\text{Fin}}$ and the consistency of $\ros_\text{1-1} < \nonstar{\EE\DD_\text{Fin}}$. To achieve this, we will take a look at a slight variation of Hechler's forcing.
\begin{center}
\begin{tabular}{l l}
$\mathbf{D} = \{ T \subseteq \bairetree :$&$T$ is a well-pruned tree with stem $s$, \\
& for every $t \in T$, if $s \subseteq t$, then $\{n \in \omega : t ^\frown \langle n \rangle \notin T \}$ is finite$\}$
\end{tabular}
\end{center}
One can easily show that $\mathbf{D}$ is a $\sigma$-centered forcing notion and that the generic real is a dominating real. We will need the following theorem, which is a direct consequence of Theorem 12 and Lemma 16 in \cite{brendlelowe}.

\begin{theorem}[Brendle, L\"{o}we]\label{brendlelowe}
    Let $\mathbf{P}_{\lambda}$ be the finite support iteration of length $\lambda$ of the forcing $\mathbf{D}$. Then, for every $\mathbf{P}_{\lambda}$-name of an infinite partial function $\dot{f}$ such that $\mathbf{P}_{\lambda} \forces{\forall n \in \mathrm{dom}f \; (\dot{f}(n) \leq n)}$ there is a sequence $\{ f_n : n \in \omega \}$ of partial functions such that $f_n(m) \leq m$ for all $m,n\in\omega$, and for every $y \in \baire \cap V$, if $f_n\cap y$ is infinite for all $n\in \omega$, then $\mathbf{P}_{\lambda}\forces{y\cap \dot{f}\text{ is infinite}}$.
\end{theorem}

We are ready to show our last result of this work.

\begin{theorem}
    Both $\Delta_\text{1-1} > \covstar{\EE\DD_\text{Fin}}$ and $\ros_\text{1-1} < \nonstar{\EE\DD_\text{Fin}}$ are consistent with ZFC.
\end{theorem}
\begin{proof}
    (Consistency of $\Delta_\text{1-1} > \covstar{\EE\DD_\text{Fin}}$): Start with a model $V$ of CH and force with $\mathbf{P}_{\omega_2}$, the finite support iteration of length $\omega_2$ of the forcing $\mathbf{D}$. Since $\mathbf{D}$ adds dominating real, then a simple reflection argument shows that $\mathbf{P}_{\omega_2}\forces{\aleph_2 \leq \mathfrak{b} \leq \Delta_\text{1-1}}$, we only have to show that $\mathbf{P}_{\omega_2}\forces{\EE\DD_\text{Fin} \cap V \text{ witnesses }\covstar{\EE\DD_\text{Fin}}}$: Let $\dot{X}$ be a $\mathbf{P}_{\omega_2}$-name for an infinite subset of $\{ \langle n,m \rangle \in \omega \times \omega : n \leq m \}$. Without loss of generality $\dot{X}$ is an infinite partial function. By Theorem \ref{brendlelowe}, there is a sequence $\{ f_n : n \in \omega \}$ of partial functions such that, for every $y \in \baire \cap V$, if for all $n\in \omega$, $f_n\cap y$ is infinite, then $\mathbf{P}_{\omega_2}\forces{y\cap \dot{X}\text{ is infinite}}$. Note that, since $\nonstar{\EE\DD_\text{Fin}}$ is uncountable, there is a function $y \in \EE\DD_\text{Fin}$ such that $f_n\cap y$ is infinite for all $n\in \omega$, so therefore $\mathbf{P}_{\omega_2}\forces{y\cap \dot{X}\text{ is infinite}}$, and as a consequence $\mathbf{P}_{\omega_2}\forces{\covstar{\EE\DD_\text{Fin} = \aleph_1}}$.
    
    (Consistency $\ros_\text{1-1} < \nonstar{\EE\DD_\text{Fin}}$): Start with a model $V$ of Martin's Axiom and $2^{\aleph_0} = \aleph_2$, and force with $\mathbf{P}_{\omega_1}$, the finite support iteration of length $\omega_1$ of the forcing $\mathbf{D}$. If $\{ \dot{d_\alpha} : \alpha \in \omega_1 \}$ is the collection of dominant reals added on each stage of the iteration, then $\mathbf{P}_{\omega_1}\forces{\{ \dot{d_\alpha} : \alpha \in \omega_1 \} \text{ is a witness of }\dom}$ and therefore $\mathbf{P}_{\omega_1}\forces{\ros_\text{1-1}\leq \dom = \aleph_1}$, so we only need to show that $\mathbf{P}_{\omega_1}\forces{\nonstar{\EE\DD_\text{Fin}} \geq \aleph_2}$: Let $\{ \dot{X_\alpha} : \alpha \in \omega_1 \}$ be a sequence of $\mathbf{P}_{\omega_1}$-names for infinite subsets of $\{ \langle n,m \rangle \in \omega \times \omega : n \leq m \}$. Without loss of generality, we may assume that each $\dot{X_\alpha}$ is a $\mathbf{P}_{\omega_1}$-name for an infinite partial function. Then, we can use Theorem \ref{brendlelowe} to find a family $\{ f^\alpha_n : \alpha \in \omega_1, n \in \omega\}$ of partial functions such that $f^\alpha_n(m) \leq m$ and if $y \in \baire \cap V$ and $f^\alpha_n\cap y$ is infinite, then $\mathbf{P}_{\omega_1}\forces{y\cap \dot{X_\alpha}\text{ is infinite}}$. Since in $V$, $\omega_1 < \nonstar{\EE\DD_\text{Fin}}$ (since Martin's Axiom holds in $V$), then there must be an infinite function $y \in \EE\DD_\text{Fin}$ such that $f^\alpha_n\cap y$ is infinite for each $\alpha \in \omega_1$ and $n\in \omega$ and therefore $\mathbf{P}_{\omega_1}\forces{\forall \alpha < \omega_1 (y\cap \dot{X_\alpha}\text{ is infinite})}$ i.e. $\mathbf{P}_{\omega_1}\forces{\{ \dot{X_\alpha} : \alpha \in \omega_1 \} \text{ is not a witness for }\nonstar{\EE\DD_\text{Fin}}}$ and therefore $\mathbf{P}_{\omega_1}\forces{\nonstar{\EE\DD_\text{Fin}} \geq \aleph_2}$.
\end{proof}

An alternative proof can be found in D. Meza's Ph.D. thesis \cite[Theorem 1.6.12]{MezaThesis}.

\bibliographystyle{alphadin} 
\bibliography{main}

\end{document}